\documentclass[12pt]{amsart}
\usepackage{amsmath}
\usepackage{amssymb}
\newtheorem{theorem}{Theorem}
\newtheorem{proposition}[theorem]{Proposition}
\newtheorem{lemma}[theorem]{Lemma}

\newtheorem*{remark*}{Remark}
\newenvironment{proof*}{\vskip 2mm\noindent {}}{\hfill $\Box$ \vskip 2mm}

\newcommand{\B}{\mathbb B}
\newcommand{\C}{\mathbb C}
\newcommand{\D}{\mathbb D}
\newcommand{\PP}{\mathbb P}
\newcommand{\R}{\mathbb R}
\newcommand{\eps}{\varepsilon}
\def\phi{\varphi}

\def\Om{\Omega}
\def\Re{\operatorname{Re}}
\def\Im{\operatorname{Im}}

\begin{document}

\title{``Convex'' characterization of linearly convex domains}

\author{Nikolai Nikolov}
\address{Institute of Mathematics and Informatics\\ Bulgarian Academy of
Sciences\\1113 Sofia, Bulgaria} \email{nik@math.bas.bg }

\author{Pascal J. Thomas}
\address{Universit\'e de Toulouse\\ UPS, INSA, UT1, UTM \\
Institut de Math\'e\-ma\-tiques de Toulouse\\
F-31062 Toulouse, France} \email{pthomas@math.univ-toulouse.fr}



\begin{thanks}{This paper was written during the stay of the second-named
author at the Institute of Mathematics and Informatics of the
Bulgarian Academy of Sciences (September 2010) supported by a
CNRS--BAS programme ``Convention d'\'echanges'' No 23811.}
\end{thanks}

\begin{abstract} We prove that a $C^{1,1}$-smooth bounded domain $D$ in $\C^n$ is
linearly convex if and only if the convex hull of any
two discs in $D$ with common center lies in $D.$

\end{abstract}

\maketitle

\section{Statements}
\label{state}

Recall that an open set $D$ in $\mathbb C^n$ is called (cf.
\cite{APS,Hor1}):

\begin{itemize}
\item {\it $\mathbb C$-convex} if any non-empty intersection with a
complex line is connected and simply connected;

\item {\it linearly convex} if its complement in $\Bbb C^n$ is a
union of affine complex hyperplanes;

\item {\it weakly linearly convex} if for any $a\in\partial D$
there exists an affine complex hyperplane through $a$ which does
not intersect $D$.
\end{itemize}

Note that the following implications hold:

\centerline{$\Bbb C$-convexity $\Rightarrow$ linear convexity
$\Rightarrow$ weak linear convexity.}

Moreover, these three notions coincide in the case of bounded
domains with $C^1$-smooth boundary (cf.~\cite{APS,Hor1}).

Let now $D$ be an open set in $\C^n,$ $z\in D$ and $X\in\Bbb
C^n.$ Denote by $d_D(z,X)$ the distance from $z$ to $\partial D$
in the complex direction $X$ (possibly $d_D(z,X)=\infty$):
$$
d_D(z,X)=\sup\{r>0:z+\lambda X\in D\hbox{ if
}|\lambda|<r\}.
$$

Note that the following three properties are equivalent:

\begin{itemize}
\item $1/d_D(z,\cdot)$ is a convex function;

\item the maximal circular open subset $D_z$ of $D$ w.r.t $z$ is convex
($1/d_D(z,\cdot)$ is the Minkowski function of $D_z-z$);

\item $D$ contains the convex hull of the union of any two (complex affine)
discs in $D$ with center $z.$
\end{itemize}

By \cite{ZZ} (see also \cite{NPZ}), any weakly linearly convex
open set has these properties.

To see this directly for the third property,
let us compute the linearly convex hull of a union of two discs, which
coincides with its convex hull.

By using linear transformations, we may reduce ourselves to the
case of $K:=(\overline \D \times \{0\}) \cup  (\{0\}\times
\overline \D)$ in $\C^2$. Then, if we identify a complex
hyperplane not passing through $0$ with the coefficients
$(a_1,a_2)$ of its representation as $\{a_1 z_1 + a_2 z_2 =1\}$,
the (polar) set of all hyperplanes not meeting $K$ is $K^*=\D^2$.
Then (see \cite{APS}) the linearly convex hull of $K$ is given by
$$(K^*)^*= \{ z : |z_1|+|z_2| \le 1 \},
$$
which coincides with the convex hull of $K$. (It is also the hull
of $K$ with respect the family of linear-fractional functions
\cite{APS}, since those are constant on complex hyperplanes).
Therefore any weakly linearly convex open set must contain the
convex hull of the union of two affine discs contained in the
domain and intersecting at their common center.

Our aim is to show that the converse is also true in the case of
$C^{1,1}$-smooth bounded domains. We do not know if this
regularity can be weakened. Non-smooth linearly convex domains can
be quite different (they can fail to be $\C$-convex), and
Aizenberg's question is still open: can any $\C$- convex domain be
exhausted by smooth $\C$- convex domains?

\begin{proposition}\label{1} Let $D$ be a $C^{1,1}$-smooth
bounded domain in $\C^n$ and let $U$ be a neighborhood of $\partial D.$ If $D$ contains
the convex hull of any two discs in $D\cap U$ with common center, then $D$ is linearly convex.
\end{proposition}

This can be considered as an analogue of the characterization of
convex domains by line segments, or of pseudoconvex domains by Hartogs figures.

More precisely, it follows from the proof of Lemma \ref{chord}
that a bounded $C^{1,1}$-smooth domain $D$ is not linearly convex
if and only if there are $c\in\partial D$ and a line segment
$[a,b]$ in the complex tangent hyperplane at $c$ such that $c$ is
its midpoint and $[a,b]\setminus\{c\} \subset D $. This is
analogous to the situation for real convexity.

\begin{proposition}
Let $U$ be a neighborhood of the boundary of a domain $D$ in
$\R^n$ such that if $D\cap U$ contains two sides of a triangle,
then it contains the midpoint of the third side. Then $D$ is
convex.
\end{proposition}

\noindent{\it Sketch of the proof.} Assume that $D$ is not convex. Then
there is a boundary point $c$ which is the midpoint of a segment
$[a,b]$ with $[a,b]\setminus\{c\}\subset D$ (cf.~\cite[Theorem
2.1.27]{Hor1}). To get a contradiction, it is enough to find $d
\in D$ near $c$ such that $[a,d]\subset D$ and $[b,d]\subset D$.
This follows by \cite[Theorem 2.1.27]{Hor1}, since we may touch
$D$ at $z$ from inside by a smooth domain not convex at $c$.
\smallskip

\section{Proofs}
\label{proofs}

\begin{proof*}{\it Proof of Proposition \ref{1}.}
First we need a slight modification of \cite[Theorem 1.4]{Hor2}.

\begin{proposition}
\label{Horrev}
Suppose that $D$ has a $C^{1,1}$-smooth boundary, with defining
function $\rho$. If for almost every $p\in\partial D$
\begin{equation}
\label{Hor13}
\liminf_{T^{\C}(p) \ni \zeta \to p} \frac{\rho(\zeta)}{|\zeta-p|^2}\ge 0,
\end{equation}
where $T^{\C}(p)$ denotes the largest complex affine subspace 
passing through $p$ and contained in the real affine tangent space to
$\partial D$, 
then $D$ is linearly convex.
\end{proposition}

Notice that \cite[Theorem 1.4]{Hor2} has the same conclusion with
slightly different hypotheses: it demands a little less boundary
regularity of $D$, but assumes the inequality \eqref{Hor13}
everywhere instead of almost everywhere. When $D$ has a
$C^2$-smooth boundary, the second partials are defined and
continuous everywhere, and so \eqref{Hor13} holds everywhere and
there is nothing more to prove.

In Section \ref{appendix}, we shall recall the steps of H\"ormander's
proof and give the small modifications needed to adapt it to our context.

Assume to get a contradiction that $D$ is not linearly convex. By
the Implicit Function Theorem, we can choose local coordinates such
that the boundary of $D$ can be written locally as a graph. By
Radema\-cher's theorem about the differentiability almost
everywhere of Lipschitz functions, applied to the first partial
derivatives of the function
defining the graph in each coordinate patch, we may find a point
$p\in\partial D$ such that $\rho$ is twice differentiable at $p$
and
\begin{equation}\label{BP}
\liminf_{T^{\C}(p) \ni \zeta \to p} \frac{\rho(\zeta)}{|\zeta-p|^2}<0.
\end{equation}
From now on, we choose such a $p.$

It is easy to show that the property we are studying
can be tested on two-dimensional subspaces,
so henceforth we assume that $\Om\subset\C^2.$

\begin{lemma}
\label{chord}
Under the above hypotheses, there exist $r>0,$ $c\ge 1$
and coordinates $(z,w)$ obtained by a complex affine transformation from the
original coordinates such that $(z(p),w(p))=(0,0)$ and
$$
\Om\cap\B_2(0,r)\supset E:=\left\{(z,w)\in\C^2:\rho_c(z,w)<0\right\}\cap\B_2(0,r),
$$
where $\rho_c(z,w)=\Re z-(\Re w)^2+c|z|^2+c(\Im w)^2.$
\end{lemma}

\begin{proof}
It will be enough to majorize $\rho$ by $\rho_c$ for some $c>0$ 
when $(z,w)$ is close enough to $p$.

First take coordinates $(z_1,w_1)$ such that
$(z_1(p),w_1(p))=(0,0)$ and the real gradient $\nabla \rho (0,0) =
(1,0)$. Then by Taylor's formula,
\begin{multline*}
\rho(z_1,w_1) = \Re z_1 + \Re \left( a_{11} z_1^2 +  a_{12} z_1
w_1 + a_{22} w_1^2 \right)
\\
+ b_{11} |z_1|^2 + \Re (b_{12} z_1 \bar w_1) + b_{22} |w_1|^2 +
o(|z_1|^2 + |w_1|^2 ),
\end{multline*}
where the coefficients $a_{ij}, b_{ij}$ are deduced from the
second order partial derivatives of $\rho$ at $(0,0)$ in the usual
way.

Here $T^{\C}(p)= \{(0,w), w \in \C\}$, and
$$
\frac{\rho(0,w_1)}{|(0,w_1)-p|^2} = \Re\left( a_{22}
\frac{w_1^2}{|w_1|^2}\right) +  b_{22} +o(1),
$$
and the $\liminf$ in \eqref{BP} is exactly $-|a_{22}|+b_{22}=:
-\ell$.

We rotate the $w_1$ coordinate so that $a_{22}=-|a_{22}|$, thus
$$
\rho(0,w_1) = -\ell (\Re w_1)^2 + (2b_{22}+\ell) (\Im w_1)^2
+o(|w_1|^2 ).
$$
To estimate the other terms,
$$
\left| a_{12} z_1 w_1 + b_{12} z_1 \bar w_1 \right| \le c_1|z_1 w_1
| \le \frac12 c_1(\eps |w_1|^2 + \frac1\eps |z_1|^2);
$$
we choose $\eps$ so that $c_1\eps \le \ell$,
$c_2:= |a_{11}| + |b_{11} |$, so
\begin{multline*}
\rho(z_1,w_1) < \Re z_1 + (c_2 + \frac{c_1}{2\eps}) |z_1|^2 -
\frac{\ell}2 (\Re w_1)^2 + (\frac{\eps c_1}2 + 2b_{22}+\ell) (\Im w_1)^2
\\+o(|z_1|^2 + |w_1|^2 )
\le \Re z_1 + c_3 |z_1|^2 - \frac{\ell}3 (\Re w_1)^2 + c_4 (\Im
w_1)^2
\end{multline*}
for $(z_1,w_1)$ small enough.

Taking $z=z_1$, $w= \sqrt{\frac{\ell}3} w_1$, we have the required
form.
\end{proof}

Further, choose $r>0$ such that $\Bbb B_2(0,r)\subset U$ and
suppose that we have two disks in $E$ of the form
$$D_1=\{ (-\delta(1-\zeta),\delta \zeta/\mu), |\zeta|\le 1\},
\quad D_2=\{(-\delta (1+\zeta),\delta \zeta/\mu), |\zeta|\le 1\},$$
where $\mu=\sqrt{2c\delta}.$ Then by taking
the midpoint of $(0, \delta /\mu)$ ($\zeta=1$ for $D_1$)
and $(0,-\delta /\mu)$ ($\zeta=-1$ for $D_2$) we find
$(0,0)$, which is a contradiction, proving Proposition \ref{1}.

One can see by changing $\zeta$ into $-\zeta$ that it is enough to check that
$D_1\subset E.$ It is clear that $D_1\subset\B(0,r)$ for any small $\delta>0.$
It remains to show that $\rho_c|_{D_1}<0$ on $D_1$. Since $c\ge 1,$ this restriction is
subharmonic and so it suffices to prove the inequality for $\zeta=e^{i\theta}$.

Dividing through by $\delta$, we have to verify
$$
-1+ \cos \theta - \frac1{2 c} \cos^2 \theta + \frac12 \sin^2
\theta + 2 c \delta(1-\cos \theta) < 0,\quad \theta \in \mathbb R,
$$
equivalently
$$
-1 +4 c \delta+ 2(1 - 2 c \delta)x  - (1+\frac1{ c})x^2
<0,\quad -1\le x \le 1.
$$
Computing the (reduced) discriminant of this quadratic polynomial
yields
$$
(1 - 2 c \delta)^2-(1+\frac1{ c}) (1 -4 c \delta) = 4
c^2 \delta^2 + 4 \delta - \frac1c <0
$$
for $\delta>0$ small enough.

The proof of Proposition \ref{1} is completed.
\end{proof*}

\section{Appendix}
\label{appendix}
\begin{proof*}{\it Proof of Proposition \ref{Horrev}.}
First, the hypothesis and conclusion of \cite[Theorem 1.4]{Hor2} are
restated in terms of the function
$$
h(z):= \inf_{w\in \partial D}|z-w|^2, \quad z \in D.
$$
For any $\zeta \in \partial D$ at which this infimum is attained,
$z$ belongs to the normal to $\partial D$ at $\zeta$. For any $z$
for which the infimum is attained at a unique point $\zeta \in
\partial D$, we denote $\pi(z)=\zeta$ this nearest point. Then $h$
is differentiable at any point $z$ where $\pi(z)$ is well defined.

The property that $\partial D$ admits a neighborhood in which the
nearest point in uniquely defined is called ``positive reach" in
\cite{Fed}. When $\partial D$ is $\mathcal C^{1,1}$ smooth (i.e.
$D$ admits a $\mathcal C^{1,1}$-smooth defining function with non
vanishing gradient near the boundary), it is of positive reach, as
follows for instance from the remark at the beginning of
\cite[Section 4]{Fed}, ``the class of sets with positive reach is
closed under bi-Lipschitzian maps
 with Lipschitzian differentials'' (or see the note added in proof at
 the end of \cite{Kr-Pa}).

It is proved in \cite{Kr-Pa} that when $\partial D$ has positive reach
and is $\mathcal C^1$, then $h \in \mathcal C^1$ and its first partial derivatives
are given explicitly in terms of the first partial derivatives of
the defining function in an appropriate
coordinate system, see equations (3) and (4) on \cite[p. 118]{Kr-Pa}.
Since the $\mathcal C^{1,1}$ property does not depend on the coordinate system,
it follows from the formulae that if $\rho $ has Lipschitz continuous first partial
derivatives, then so does $h$, i.e.
$h \in \mathcal C^{1,1}$.
Note that for any $\eps>0$ there exist
domains with  $C^{2-\eps}$-smooth boundary which do not have positive reach,
and where $h$ fails to be differentiable in a neighborhood of the boundary,
so $\mathcal C^{1,1}$ smoothness is a kind of minimal hypothesis we have to
demand using these methods.

It is also noted in a remark after the statement of \cite[Theorem
1.4]{Hor2} that when $\partial D$ is $\mathcal C^{1,1}$-smooth,
the regularity hypotheses of the theorem, i.e. the interior ball
condition (which implies positive reach) and condition
\cite[(1.5)]{Hor2}, are in particular fulfilled.

\cite[Proposition 1.5]{Hor2} shows that $D$ is linearly convex if and only if
\begin{equation}
\label{Hor16}
h(w) \le h(z) + 2 \Re \langle w-z,h'_z(z)\rangle +
\left| \langle  w-z,h'_z(z)\rangle \right|^2/h(z), w \in D,
\end{equation}
for every $z$ at which $h$ is differentiable.

Then the proof of \cite[Theorem 1.4]{Hor2}, given at the end of
\cite[Section 2]{Hor2}, proceeds as follows: given a $\delta>0$
such that $h$ is $\mathcal C^1$ on $\D_\delta := \{ h < \delta^2 \}$,
choose for each $\zeta \in \partial D$ a ball $B_{\delta, \zeta}$
of radius $\delta/2$, tangential to $\partial D$ at $\zeta$ (and thus
contained in $\D_\delta$). It is shown that $D$ is linearly convex if
\eqref{Hor16} holds on each $B_{\delta, \zeta}$. For that it is enough
to show that $h$ is \emph{quadratically concave}, a notion defined by
\cite[Definition 2.1]{Hor2}.

\cite[Theorem 2.4]{Hor2} shows that when a positive function $g$ is
defined on a ball $B\subset R^N$ and its second derivatives are measures,
a sufficient
condition for it to be quadratically concave is that
\begin{equation}
\label{Hor26}
\langle g''(z) v, v \rangle \le \frac12 |v|^2 |g'(z)|^2/g(z) ,
\end{equation}
for any
$z \in B$, $v \in \mathbb R^N$.

We need to show that the conclusion of \cite[Theorem 2.4]{Hor2}
still holds under our new hypotheses.

\cite[Proposition 1.6]{Hor2}
shows that \eqref{Hor13} at $\zeta \in \partial D$
implies
\begin{equation}
\label{Hor17}
\limsup_{w\to 0} \frac1{|w|^2}
\left( h(z+w) - h(z) - 2 \Re \langle w, h'_z(z) \rangle
- \frac{|\langle w, h'_z(z) \rangle|^2}{h(z)} \right) \le 0
\end{equation}
 at any point $z$
at which $h$ is differentiable,
 and such that $\zeta=\pi(z)$. Since $z$ must lie in the normal
 to $\partial D$ at $\zeta$ and $\partial D \in \mathcal C^{1,1}$,
 the implicit function theorem shows that if $N \subset \partial D$
 is of $(2n-1)$-Lebesgue measure zero, then in a neighborhood of
 $\partial D$,  $\pi^{-1}(N)$ is of $2n$-Lebesgue measure zero.
So our hypothesis implies that \eqref{Hor17} holds for almost every $z$
in a neighborhood of the boundary. Applying Taylor's formula,
we see that this implies that \eqref{Hor26} holds for a.e. $z$
in this neighborhood and any vector $v$.  So our hypotheses
differ from H\"ormander's only in that the second derivatives
of our function are measures given by a.e. defined bounded functions,
and by requiring a second-order differential inequality almost everywhere
instead of everywhere.

\cite[Theorem 2.4]{Hor2} shows that $g$ is quadratically concave by
proving that
\begin{equation}
\label{Hor22}
g(y) \le g(x) + \langle y-x, g'(x) \rangle + \frac14 |y-x|^2
\frac{|g'(x)|^2}{g(x)}, x,y \in B,
\end{equation}
which is enough by \cite[Theorem 2.2]{Hor2}.

Assuming that $B$ is the unit ball, if we prove \eqref{Hor22}
for $g(x)=h_r(x):=h(x)$, and $0<r<1$, we obtain the same inequality
for $h=h_1$, using the fact that $h\in \mathcal C^1$ to pass to the
limit.  So we may now assume that $g\ge \eps_0>0$. Then we follow
part (c) of the proof of \cite[Theorem 2.4]{Hor2}: writing
$g_\eps (x) := g(x)- \eps (|x|^2+1)$ for small enough $\eps>0$,
we have that
\begin{equation*}
\langle g_\eps'' t,t \rangle \le
\left( \frac12 |g_\eps'(x)|^2/g_\eps(x) - 2\eps/(|x|^2+1) \right) |t|^2,
\end{equation*}
and so a $\mathcal C^\infty$ regularization of $g_\eps$ will
satisfy \eqref{Hor26} with strict inequality in $\overline B$ when
$|t|=1$ (since the second order differential inequality is
obtained by integrating against the regularizing kernel bounded
functions that satisfy the inequality almost everywhere). By parts
a) and b) of the proof of \cite[Theorem 2.4]{Hor2} (which are
valid for $\mathcal C^2$-smooth functions), this regularization
will satisfy \eqref{Hor22}, and letting the regularization tend to
$g$, we obtain \eqref{Hor22} for $g$.

\end{proof*}

\end{document}